\documentclass[12pt,letterpaper,leqno]{article}
\usepackage[a4paper, hmargin=2.5cm, vmargin={2.5cm, 2.5cm}]{geometry}
\usepackage[fontsize=11pt]{scrextend}
\usepackage{enumitem}
\usepackage{amsfonts}
\usepackage{amssymb}
\usepackage{amsmath}
\usepackage{amsthm}
\usepackage{graphicx}
\usepackage{hyperref}
\usepackage{xcolor}

\newcommand{\calE}{\mathcal{E}}

\def\R{\ensuremath{\mathbb R} }
\def\C{\ensuremath{\mathbb C} }

\def\N{\ensuremath{\mathbb N} }
\def\Z{\ensuremath{\mathbb Z} }

\def\eps{\ensuremath{\varepsilon} }
\def\vect{\ensuremath{\overset{\rightharpoonup}}}
\newcommand{\GR}{G(\mathcal{R})}

\newcommand* \bb [1]{\left({#1}\right)}
\newcommand* \sbb [1]{\left[{#1}\right]} 
\newcommand* \bset [1]{\left\{{#1}\right\}} 
\newcommand* \limit [2]{\underset{{#1}\rightarrow{#2}}{\lim}\;}
\newcommand* \bunion[3]{\bigcup_{{#1} = {#2}}^{#3}}

\newcommand* \bintersect[3]{\bigcap_{{#1} = {#2}}^{#3}}
\newcommand* \sumit [3]{\underset{{#1} = {#2}}{\overset{#3}\sum}\;}
\newcommand* \prodit [3]{\underset{{#1} = {#2}}{\overset{#3}\prod}\;}
\newcommand* \abs[1] {\left|{#1}\right|}
\newcommand* \norm[2]{\left|\left|{#1}\right|\right|_{#2}}
\newcommand* \indic [1]{\mathbf 1_{#1}}

\newtheorem{thm}{Theorem}[section]
\newtheorem{lem}[thm]{Lemma}

\newtheorem{prop}[thm]{Proposition}
\newtheorem{cor}[thm]{Corollary}
\newtheorem{defn}[thm]{Definition}
\newtheorem{rmk}[thm]{Remark}

\begin{document}
\title{Measurable entire functions II}

\author{Adi Gl\"ucksam and Benjamin Weiss}

\maketitle

 \begin{abstract}
Let $\mathcal{E}$ denote the space of entire functions with the topology of uniform convergence on compact sets. The action of $\C$ by translations on $\calE$ is defined by $T_zf(w) = f(w+z)$.
Let $\mathcal{U}$ denote the set of entire functions whose orbit under $T$ is dense. Birkhoff showed, in \cite{B}, that $\mathcal{U}$ is not empty.  One of the problems in the collection by T-C Dinh and N. Sibony \cite{DS} asks whether there exists an invariant probability measure on $\mathcal{E}$ whose support is contained in $\mathcal U$. We will show how an old construction of the second author can be modified to provide a positive answer to their question. Furthermore, we modify the construction to produce a wealth of ergodic measures on the space of entire functions of several complex variables.
\end{abstract}

\section*{Introduction}
   The complex plane $\C$ acts on the space of entire functions $\calE$ by translations, for $ f \in \calE$,
   $T_zf(w) = f(w+z)$. With the topology of uniform convergence on compact sets, $\calE$ is a Polish space and the translation action is continuous.
   In \cite{B} Birkhoff constructed an entire function with a dense orbit, showing that this translation action is topologically transitive. This is equivalent
to the fact that there is in fact a dense $G_{\delta}$ set of entire functions whose orbits are dense (see for example \cite{GH} Theorem 9.20).
   Answering a question of George Mackey, in \cite{W}, the second author constructed many invariant ergodic measures on $(\calE, T_z)$. In their recent collection of open problems
   \cite{DS} Dinh and Sibony asked can there be such a measure which is concentrated on functions whose orbit is dense.

      In the following note we will first show how to modify a construction in \cite{W} to achieve this goal.
       The measure $\mu$
      will be constructed using a simple ergodic action of $\C$ by translation on a compact group $G$ via a dense subgroup $H$ which is a copy of $\C$. This action
      preserves the Haar measure $\rho$ and since $H$ is dense the action is ergodic. We will construct a complex valued function $F$ on $G$ such that for a.e. point $g_0 \in G$
      the restriction of $F$ to $Hg_0$ defines an entire function in $\mathcal{U}$, i.e., its orbit under the translation map is dense. The measure $\mu$ is then
      simply the image of Haar measure under this mapping. More formally, the mapping $\hat{F}$ from $g$ to $\calE$ is given by $\hat{F}(g)(z) = F(g+z)$, where we have
      switched to an additive notation and identified $H$ with $\C$. As the image of an ergodic measure we are assured that $\mu$ will ergodic. The special properties
      of the function $F$ that we will construct will guarantee the desired result.

      In the next section we describe the group $G$ and the tools from ergodic theory  that we will need while in the second section the construction of $F$ will be given.
      In the third section we show how to extend this construction to the space of entire functions of several complex variables.

      The final sections are devoted to showing how to go from the translation action on the compact group $G$ (the solenoid) to any free action of $\mathbb{C}^d$ as the sample space for the measures on the space of entire functions.


\section{Solenoids and their towers}
   For the sake of clarity we will first briefly review a simple analogue of the compact group $K$ and the towers related to $H \triangleleft K$ that we will need.
   Denote by $[r]$ the group $\Z / r\Z$ (we will denote the elements of $[r]$ by their representatives $\{0, 1, \cdots r-1 \}$) , and then
   for a sequence of integers $\R = \{ r_1, r_2, \cdots r_n , \cdots \}$ define $R_n = \prod_{i=1}^n r_i$. There is a natural homomorphism from $[R_{n+1}]$ onto $[R_n]$ and the inverse
   limit of this series of groups, $G(\R)$, is called a generalized odometer. It is naturally identified with $ X = \prod_{i=1}^{\infty} [r_i]$ and the element
   $u$ defined by $u(1) =1$ and $u(i) =0$ for all $i >1$,  generates a copy of $\Z$ where the positive elements are those with only a finite number of coordinates different from $0$.
    If all the $r_i = 2$, this is the group of $2$-adic integers. The action defined by translation by $u$, in this case, is called the dyadic odometer.

    The Haar measure $\rho$ on the group $\GR$ is the product of the uniform measures on each $[r_i]$. As is usual in ergodic theory we will denote translation by $u$ by $T$.
    Denote by $(X, \rho)$ this measure space and as is usual in ergodic theory we will denote translation by $u$ by $T$. In the system $(X, \rho, T)$
    define $B_n = \{ x : x(i) = 0, \hspace{0.2cm}\textrm{ for all } 1 \leq i \leq n \}$. It is clear that the sets $T^iB_n$ are disjoint for all $0 \leq i \leq R_n -1$ and their union is all of $X$. We call
    this partition of $X$ a tower over $B_n$.  It
    is also clear that $\rho(B_n) =\frac{1}{ R_n}$, and that the tower over $B_{n+1}$ is divided into $r_{n+1}$ translates of subsets of $B_n$. There is a simple algebraic description of these sets,
   namely
    by its definition as an inverse limit there are homomorphisms $\pi_n: \GR \rightarrow [R_n]$ and the $B_n$'s are the kernels of these mappings.

    In exactly the same manner one can define a continuous analogue of this construction replacing $\Z$ by $\R$. The finite cyclic groups are replaced by circles $\R / R_n \Z$ and the resulting
    inverse limit is called a \textbf{solenoid}.  It contains a copy $H$ of $\R$ as a dense subgroup. Once again we will denote by $\pi_n$ the continuous homomorphisms from the solenoid onto the circles
    $\R / R_n \Z$.

    There is a useful way to view the dynamics of addition by $H \cong \R$ on the solenoid.
    We first mark partitions of the line $\R$ into adjacent intervals $[t + ir_1, t+ (i+1)r_1)$ where $t$ is any point in $[0, r_1)$ and $i \in \Z$.
     Next these intervals are formed into consecutive groups of size $r_2$ ,
    and these are marked by an additional mark. This produces intervals $[t + iR_2, t+ (i+1)R_2)$  of length $R_2$ where  $t$ now ranges over $[0, R_2)$. This process is continued by induction. The elements of
    the solenoid are represented by this hierarchal sequence of partitions and the action by $\R$ is simply translation of this sequence of partitions. What correspond to the bases of the towers $B_n$
    are now those partitions which in stage $n$ have $t=0$. As before this just the kernel of $\pi_n$.  This is a compact subgroup with its own
     Haar measure which when multiplied by the Lebesgue measure
    on $[0, R_n)$ gives the Haar measure on the solenoid.

    Exactly the same construction can be carried out for $\R^2$, using the sequence $R_n\Z^2$.
     The resulting $2$-dimensional solenoid $G$, containing an $H \cong \R^2$, will define the ergodic
    system that we need. The new projections $\pi_n: G \rightarrow \R^2 / R_n\Z^2$ will define the sets $B_n = \textrm{ker}(\pi_n)$ that will be the basis for our construction. By $S_n$ we will denote
    the square $[0, R_n)^2$ which is a fundamental domain for $\R^2 / R_n\Z^2$. Note that $S_{n+1}$ is partitioned into $r_{n+1}^2$ translates of $S_n$ and we will denote
    these as follows:
      $$ S_{n,ij} = [iR_n, (i+1)R_n) \times [jR_n, (j+1)R_n) \hspace{0.1cm}   \textrm{for} \hspace{0.1cm} 0 \leq i , j < r_{n+1}. $$

     For each $n$ the group $G$ can be identified with $ S_n\times B_n$ and $\rho$, Haar measure of $G$,
    with the product of the Haar measure on $B_n$ with Lebesgue measure on $S_n$.
    The group $G$ depends
    of course on the sequence $\{r_n \}$, which can be arbitrary, the only requirement that will be needed is that the sum of the series $\sum_1^{\infty} \frac{1}{r_n}$ is finite.

   \section{The construction}
   The tool from the theory of complex variables that we will need is Runge's theorem which we now recall.

   \begin{thm}(C. Runge, 1985)
   If $K$ is a compact subset of $\C$ with a connected complement and $f$ is holomorphic in some neighborhood of $K$ then given $\eps > 0$ there is a polynomial $p$
   such that
   $$ \sup_{z \in K} |f(z) - p(z)| < \eps.$$
   \end{thm}

     To illustrate its strength here is a simple proof of Birkhoff's result. In order to show that the flow $( \calE, T_z) $ is topologically transitive given two open sets
     $U_0, U_1$ in $\calE$ we must find a point $f_0 \in U_0$ and some $z_0 \in \C$ such that $T_{z_0}f_0 \in U_1$. Both  $U_i$ contain  sets $V_i$  defined by polynomials $p_i$, and constants
     $(M, \eps)$ of the form $$V_i = \{f \in \calE : \sup_{|z| \leq R} |f(z) - p_i(z)| < \eps \}.$$

     Now  define a compact set $K$ as the union of the two disks $ D_0 = \{ z: |z| \leq R \}$ and $D_1 = \{ z: |z-3R| \leq R \}$, and in a neighborhood of $K$ define a holomorphic function $f$ as
     $p_0(z)$ on neighborhood of $D_0$ and $p_1(z-3R)$ on a neighborhood of $D_1$. Applying Runge's theorem there is a polynomial $p$ within $\eps$ of this holomorphic function $f$ on $K$. This
     polynomial, which is clearly in $\calE$  is an element of $V_0$ and when translated by $3R$ is in $V_1$ as required.

     \begin{rmk} There is no result as general as Runge's theorem in several complex variables. However there is such a result for the disjoint union of two compact convex sets which suffices to extend
     the above proof to the space of entire functions on $\C^n$. Details can be found for example in \cite{B-G}.
     \end{rmk}

   We fix now a sequence $\{r_n \}$ and let $G$ denote the group that was defined in the previous section. We also fix a sequence of polynomials $\{p_n \}$ that are dense in $\calE$. In each stage
   we will define a complex valued function $F_n$ and these will converge to the desired function. While we defined the sets $S_n$ as subsets of the Euclidean plain we will also
   view them as subsets of the complex plane and freely denote their elements using complex variables. To define $F_1$ on $G$ set $G = S_1\times B_1$ and set $F_1( z , b) = p_1(z)$.

   In stage 2  we first define compact subsets
   $K_{1,ij} \subset S_{1, ij}$ for $0 \leq i, j < r_2$ by
$$K_{1,ij} = \left[iR_1 + \frac{1}{10}, (i+1)R_1 -\frac{1}{10}\right] \times \left[jR_1 + \frac{1}{10}, (j+1)R_1 -\frac{1}{10}\right]. $$
     In the tiling of $S_2$ by  $r_2^2$ copies of $S_1$ these are disjoint squares strictly contained in the original tiles.
   Now view $G$ as $ S_2\times B_2$. Restrict $F_1$ to the union of the $ K_{1,ij}\times B_2$. In addition to this on $ K_{1,00}\times B_2$ replace $F_1$ by $p_2$. Now apply Runge's theorem to this function,
   which is clearly holomorphic on a neighborhood
   of the union of the $  K_{1,ij}\times B_2$,  with $\eps = \frac{1}{10}$ and this defines $F_2$, which is a polynomial, on  $ S_2\times \{ w \}$ for each $ w \in B_2$.

   For the induction we assume that we have defined $F_{n-1}$
   which is a polynomial when restricted to $S_{n-1}\times \{  w \}$ for each $ w \in B_{n-1}$.
   Next  define compact subsets $K_{n-1,ij} \subset S_{n-1, ij}$ as follows:
    $$K_{n-1,ij} = \sbb{iR_{n-1} + \frac{1}{10^{n-1}}, (i+1)R_{n-1} -\frac{1}{10^{n-1}}} \times \sbb{jR_{n-1} + \frac{1}{10^{n-1}}, (j+1)R_{n-1} -\frac{1}{10^{n-1}}}. $$
  Now view $G$ as $S_n\times B_n $. Restrict $F_{n-1}$ to the union of the $ K_{n-1,ij}\times B_n$ and on $ K_{n-1,00}\times B_n$ replace $F_{n-1}$ by $p_n$. Apply now Runge's theorem to approximate
  this new function on the union the $ K_{n-1,ij}\times B_n$ with $\eps = \frac{1}{10^{n-1}}$ to obtain a polynomial $F_n$. This defines $F_n$ on $G$, it is of course a polynomial on
  $S_n\times\bset w$ on each $ w \in B_n$.

Let $D_R:=\bset{\abs z<R}$, then for a.e $g$ for every $n$ large enough $D_Rg:=\underset{ z\in D_R}\bigcup T_zg\subset  \bigcup_{ij} K_{n,ij}B_n$, where the union extends over all $1 \leq i,j \leq r_n -1 $. This shows that for a.e $g$ the functions $F_n^g:D_R\rightarrow \C$ defined by $F_n^g(z)=F_n(T_zg)$, when restricted to $Hg$, converge uniformly on compact subsets of $Hg$ to an entire function.

   The translates of this function will contain portions that are
   within $\sum_{m \geq n-1} \frac{1}{10^m}$ of $p_n$ on large neighborhoods of 0. This shows that for $\rho$-a.e. $g$ the limiting function $F$ restricted to its $H$ orbit is an element of
   $\mathcal{U}$ as required. The required measure $\mu$ on $\mathcal{E}$ is defined as the push forward of $\rho$ by the mapping $F$.  This completes the proof of our main theorem:

   \begin{thm} Let $G$ be a solenoidal group as described in section 1, denote by $\rho$ its Haar measure and by $H$ the dense subgroup which is isomorphic to $\C$. There is a complex
   valued function $F$ on $G$ which when restricted to $\rho$-a.e. coset of $H$ is an entire function with a dense orbit under the translation action of $\C$ on $\mathcal{E}$. The push forward of the measure $\rho$ by the mapping $F$, $F\circ\rho$, is an ergodic measure $\mu$ on $\mathcal{E}$, and $\mu$-a.e $f$ belongs to $\mathcal{U}$, the entire functions with dense orbits.
   \end{thm}
\section{Several complex variables}

   There is no general analogue of Runge's theorem in several complex variables. Nonetheless, somewhat surprisingly, there is such an approximation theorem
   for the generalizations to higher dimensions of the compact sets that were involved in the previous construction. Since this is not so well known we will sketch
   a proof of the result that we need. In this and the following sections we will denote elements of $\C^d$ by $\vect z$, i.e. $\vect z = (z_1, z_2, \cdots, z_d) $.   We need the following definition:

   \begin{defn}The  \textbf{polynomially convex hull} $\hat{K}$ of a compact subset $K \subset \C^d$ is defined by:
   $$ \hat{K} = \{ \vect w : |p(\vect w)| \leq \sup_{\vect z \in K}|p(\vect z)| \hspace{0.1 cm}  \textrm{for all polynomials} \hspace{0.1cm}  p \} .$$
   The compact set $K$ is called \textbf{polynomially convex} if $\hat{K} = K$.
   \end{defn}

    In one dimension the condition for a compact set to be polynomially convex is that its complement is connected. In higher dimensions there is no simple
    criterion. There is an approximation theorem for polynomially convex sets in higher dimensions
    called the \textbf{Oka-Weil Approximation Theorem}. It asserts that if a compact set $K \subset \C^d$  is polynomially convex then any complex valued function $f$
    that is holomorphic in a neighborhood of $K$ can be approximated uniformly
   on $K$ by polynomials (see \cite{G} III.5). The simplest polynomially convex sets are the convex sets.  Eva Kallin \cite{K} found a simple condition for
   the union of two disjoint polynomially convex sets to be polynomially convex. Here is her result:

   \begin{lem}\label{lem:Kallin}(E. Kallin 1964) Let $K_1$ and $K_2$ be compact subsets in $\C^d$ and assume  that there exists
   a polynomial $p$ satisfying $\widehat{ p(K_1)} \cap \widehat{ p(K_2) } = \emptyset$, then
   $\widehat{ K_1 \cup K_2 } = \hat{K_1} \cup \hat{K_2} $.
   \end{lem}

   In particular if  $K_1, K_2$ are  convex  their union is polynomially convex since they can be separated
   by a one degree polynomial.
    In order to simplify the notation we will now set $d=2$, the general case needs no new ideas. Here is the proposition
   that we will need.
   \begin{prop}\label{products}
   Suppose we have two collections of pairwise disjoint compact subsets in $\C$, $\{ K_1, K_2, \cdots K_a \}$ and $\{ L_1, L_2, \cdots L_b \}$ such
    that the complements $\C \backslash \bigcup_{i=1}^aK_i$ and $\C \backslash \bigcup_{j=1}^b L_j$ are connected. Assume in addition that each product $K_i \times L_j$
    is polynomially convex. Then the union
    $$\bigcup_{i=1,j=1}^{a,b}K_i \times L_j$$ is polynomially convex.
   \end{prop}
   \begin{proof} The proof is by induction. Let $b=1$ and set $\mathcal K_u = \bigcup_{i=1}^u K_i$.
   Suppose that for $u < a$  we have already shown that $\mathcal K_u  \times L_1$ is polynomially convex. By the hypothesis this holds for $u=1$, the base of the induction. By
   Runge's theorem we can find a polynomial $p$ of one variable such that $sup_{z \ {\mathcal K}_u} |p(z)| \leq 0.1 $ and $sup_{z \ K_{u+1}} |p(z)-1| \leq 0.1 $. We can view $p$ as a polynomial \textbf{p} on $\C^2$
   by setting $ \textbf{p }(\vect z) = p(z_1)$ and then apply Kallin's lemma, Lemma \ref{lem:Kallin}, to the sets $\mathcal K_u \times L_1$ and $K_{u+1} \times L_1$ to conclude that the set
   $$
   \mathcal K_u\times L_1 \cup K_{u+1} \times L_1={\mathcal K}_{u+1}\times L_1
   $$
   is also polynomially convex.
   By induction we  then conclude that $\mathcal K_a \times L_1$ is polynomially convex. Now the same argument shows that each of the sets $\mathcal K_a \times L_j$ for $1 \leq j \leq b$ is polynomially
   convex. Repeating the inductive argument
   using  the second coordinate will conclude the proof. \qedsymbol
   \end{proof}

     With this proposition a construction analogous to the one in the previous section can be carried out in several complex variables. In general, we would identify $\C^d$ with $\R^{2d}$ and
construct the $2d$-dimensional solenoid as an inverse limit of $\R^{2d} / R_n\Z^{2d}$ with the sequence $R_n$ as before. It is in fact $G^d$ where $G$ is the 2-dimensional solenoid found above.
   With our previous notation $B_n^d$ is now a compact subgroup and $G^d$ can be identified with $B_n^d \times S_n^d$. In the inductive step of the construction the use of Runge's theorem
is now replaced by Oka-Weil Approximation Theorem. The use of this theorem is possible since the compact sets where we want to approximate $F_{n-1}$ are products of compact sets in one complex variable of the form $\bigcup_{j,i} K_j\times L_i$, and therefore, following Proposition products, the union is polynomially convex. We will leave the details to the reader and simply formulate the result. Denote by $\mathcal{E}(d)$ the entire functions on $\C^d$ and by $\mathcal{U}(d)$ those functions with a
   dense orbit under the translation action of $\C^d$ on $\mathcal{E}(d)$.

   \begin{thm} Let $G$ be a solenoidal group as described in section 1. The Haar measure on $G^d$ is $\rho^d$ and  $H^d$ is the dense subgroup which is isomorphic to $\C^d$. There is a complex
   valued function $F$ on $G^d$ which when restricted to $\rho^d$-a.e. coset of $H^d$ is an entire function on $\C^d$ with a dense orbit under the translation action of $\C$ on $\mathcal{E}$. The push forward of the measure $\rho^d$ by the mapping $F$,
   $F\circ\rho^d$, is an ergodic measure on $\mathcal{E}(d)$ with support in $\mathcal{U}(d)$.
   \end{thm}


\section{Almost Polynomially Convex Sets}
To extend the proof for general free actions of $\C^d$, we will need to show that a finite union of cubes, which are no longer necessarily of the form $\bigcup_{i,j} K_j\times L_i$, is also polynomially convex if one is willing to remove small parts of the set, i.e., it is almost polynomially convex. The goal of this section is to show that any collection of unit cubes forms an almost polynomially convex set, see Proposition \ref{prop:poly_conv}. The ideas of the proposition and its proof are reminiscent of  \cite[Lemma 2.1]{AG2019}. To state the proposition rigorously, and prove it and the supporting lemmas we begin with explaining the notation and definitions we will use in this section.
\subsection{Definitions, Notation, and the Main Lemma}
From now on $q$ will denote some vector in $\C^{d}\cong \R^{2d}$ and $Q_0:=\sbb{-\frac12,\frac12}^{2d}$. We will say a cube $Q$ is {\bf parallel to the axes} if the facets of $Q$ are parallel to the hyperplanes $\bset{x_k=0}$. In particular, any cube of the form $Q=q+Q_0$ is a unit cube parallel to the axes, centred at $q$.

Let $Q=q+Q_0$ be a unit cube parallel to the axes. Partitioning every segment defining this cube into two segments of equal length, we create the collection $\bset{Q^{\vect i}}_{\vect i\in\bset{-1,0}^{2d}}$ defined by
$$
Q^{\vect i}=q+\prodit j 1 {2d} \sbb{\frac{i_j}2,\frac{i_j+1}2},
$$
where, for every $j$
$$
\sbb{\frac{i_j}2,\frac{i_j+1}2}=	\begin{cases}
							\sbb{-\frac{1}2,0}&, i_j=(-1)\\
							\sbb{0,\frac{1}2}&, i_j=0
						\end{cases}.
$$
We will call every sub-cube, $Q^{\vect i}$, a {\bf $\frac1{2^{2d}}$ sub-cube of $Q$}.

The collection of all unit cubes centered on $\Z^{2d}$ is called {\bf the grid}, and every cube in the grid is of the form, $Q=q+Q_0, q\in\Z^{2d}$, and is called a {\bf grid-cube}.

\begin{figure}[ht]
	\centering
	\includegraphics[scale=0.85]{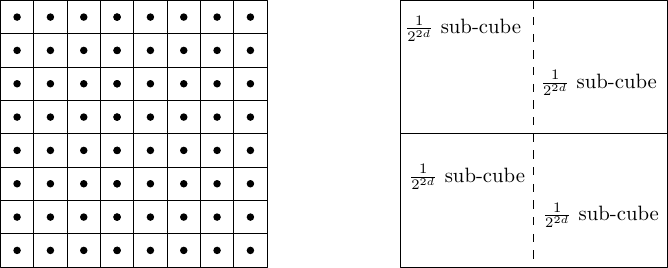}
	\caption{The figure to the left demonstrates grid cubes, the dots are the lattice $\Z^{2d}$. The figure to the right shows all the $\frac1{2^{2d}}$ sub-cubes of a cube in $\C\cong\R^2$.}
	\label{fig:partition}
\end{figure}

Given $q\in\R^{2d}$, we denote by $\sbb q\in\Z^{2d}$ the point from $\Z^{2d}$ closest to $q$, i.e.,
$$
\sbb {q}_j=\begin{cases}
				\lceil q_j\rceil&, q_j-\lfloor q\rfloor>\frac12\\
				\lfloor q_j\rfloor&,q_j-\lfloor q\rfloor\le\frac12
			\end{cases}.
$$

\begin{prop}\label{prop:poly_conv}
For every $\eps>0$ and any finite collection of pairwise disjoint unit cubes parallel to the axes, $\bset{Q_j}_{j=1}^M$, there exists a polynomially convex set $U\subset\bunion j 1 M Q_j$ satisfying that for every $j$, there exists a $\frac1{2^{2d}}$ sub-cube of $Q_j$, $B_j$, so that $B_j^{-\frac{\eps}{M\cdot  2^{7d}}}\subset U$ and $m_{2d}\bb{\bunion j 1 M Q_j\setminus U}<\eps$, where
$$
B_j^{-\frac{\eps}{M\cdot  2^{7d}}}:=\bset{\vect z\in B_j, dist(\vect z,\partial B_j)\ge\frac{\eps}{M\cdot  2^{7d}}}.
$$
If $Q_j$ was a grid cube, then $Q_j^{-\frac{\eps}{M\cdot  2^{7d}}}\subset U$.
\end{prop}


\subsection{On Polynomial Convexity and Hyperplanes}
To prove Proposition \ref{prop:poly_conv}, we require several geometric lemmas, which we prove in this subsection.


\begin{lem}\label{lem:intersect}
Every closed unit cube parallel to the axes intersects $\Z^{2d}$.
\end{lem}
\begin{proof}
Let $Q$ be a unit cube parallel to the axes, i.e., $Q=q+Q_0$. Let $\sbb q$ be the point from $\Z^{2d}$ closest to $q$. Since for every coordinate, by definition, $\abs{q_j-\sbb{q}_j}\le\frac12$ and $q$ is the center of the unit cube, $Q$, then $\sbb q\in Q\cap\Z^{2d}$.
\end{proof}


\begin{lem}\label{lem:grid}
For every unit cube parallel to the axes, $Q$, there exists a grid-cube, $G$, so that $G\cap Q$ contains a $\frac1{2^{2d}}$ sub-cube of $Q$.
\end{lem}
\begin{rmk}
Note that if $Q$ is centered exactly at the intersection of $2^{2d}$ grid-cubes, then all of them contain a $\frac1{2^{2d}}$-sub-cube of $Q$.
\end{rmk}
\begin{proof}
Let $Q$ be a unit cube parallel to the axes, i.e., $Q=q+Q_0$. We define the cube $G=\sbb q+Q_0$.

It is enough to show that the cube $G$ contains both the center of $Q$ and one of its vertices to conclude it must contain a $\frac1{2^{2d}}$ sub-cube of $Q$. This is since the intersection of two closed boxes parallel to the axes is always a closed box parallel to the axes.


Assume that $G$ does not contain any vertex of $Q$. Following Lemma \ref{lem:intersect}, $\sbb q\in G\cap Q$ and so their intersection cannot be empty, then $G$ must be contained in $Q$. However, since they share the same edge-length, $Q$ cannot contain $G$, unless $G=Q$ and, in particular, $G$ must contain a vertex of $Q$.
\end{proof}


\begin{lem}\label{lem:hyper}
Let $K_1,K_2\subset\C^{d}$ be two polynomially convex sets, and assume that there exists a hyperplane separating them, then $K_1\cup K_2$ is polynomially convex.
\end{lem}
\begin{rmk}
While the proof uses Runge's theorem, it only uses it for one variable and therefore it does not restrict the dimension, $d$.
\end{rmk}
\begin{proof}
Let $H$ be the hyperplane separating the sets $K_1$ from $K_2$, and let $L:\R^{2d}\rightarrow\R$ be a linear map satisfying $H=\bset{L=0}$. While $L$ is a linear map over $\R^{2d}$ it can be viewed as a map (but not linear) over $\C^d$, i.e., if $\vect z=(z_1,\cdots,z_d)\cong\bb{x_1,y_1,x_2,y_2,\cdots,x_d,y_d}\in\R^{2d}\cong\C^{d}$, then we abuse the notation of $L$ and define $L:\C^d\rightarrow\R$ by $L(\vect z)=L\bb{x_1,y_1,x_2,y_2,\cdots,x_d,y_d}$.

Since $H$ separates the sets $K_1$ from $K_2$, and both sets are compact, there exists $\eps>0$ so that $K_1\subset \bset{L(\vect z)>\eps}$ and $K_2\subset\bset{L(\vect z)< -\eps}$ (if the reverse holds the linear map $(-L)$ satisfies those requirements).

Note that polynomial convexity is invariant under rotations and translations. This holds since the set of polynomials is invariant under compositions with translations and rotations. We may therefore assume without loss of generality that $L(\vect z)=\Re(z_1)=x_1$.

Denote by $A_j$ the projection of the set $K_j$ onto the complex plane, i.e., $z\in A_j$ if there exists $\vect z\in K_j$ so that $z_1=z$. Note that $A_1,A_2$ are compact sets as projections of compact sets. Moreover if $z\in A_j$ denotes the projection of $\vect z\in K_j$, then
$$
dist(A_1,A_2)=\underset{z\in A_1\atop w\in A_2}\sup\;\abs{z-w}\ge\underset{z\in A_1\atop w\in A_2}\sup\;\abs{\Re(z-w)}=\underset{\vect z\in K_1\atop \vect w\in K_2}\sup\;\abs{L(\vect z)-L(\vect w)}\ge \eps-(-\eps)=2\eps.
$$
Let $p_0:\C\rightarrow\C$ be a polynomial satisfying $\abs{p_0(z)}<\frac13$ for every $z\in A_1$ and $\abs{p_0(z)-1}<\frac13$ for every $z\in A_2$. Such a polynomial can be constructed by applying Runge's theorem to the model map assigning $0$ to every $z\in A_1$ and $1$ to every $z\in A_2$ with error $\frac13$.

Define the polynomial $p(z_1,z_2,\cdots,z_d):=p_0(z_1)$. Then
$$
\widehat{p(K_1)}\cap \widehat{p(K_2)}\subset\widehat{\bb{-\frac13,\frac13}}\cap\widehat{\bb{\frac23,\frac43}}=\emptyset,
$$
and following Kallin's lemma, Lemma \ref{lem:Kallin}, $\widehat{K_1\cup K_2}=\widehat{K_1}\cup\widehat{K_2}=K_1\cup K_2$, i.e., $K_1\cup K_2$ is polynomially convex.
\end{proof}


\begin{prop}\label{prop:grid}
Let $K\subset\C^{d}$ be a compact set which does not intersect the boundary of any grid-cube. If for every grid-cube, $G$, the set $K\cap G$ is polynomially convex, then $K$ is polynomially convex.
\end{prop}
\begin{proof}
We will prove this proposition by induction on the dimension. Note that since $K$ is compact, then there exists $N$ so that $K\subset [-N,N]^{2d}$.

{\bf Step 1:} For every $\bb{j_2,j_3,\cdots,j_{2d}}\in\Z^{2d-1}$ fixed, and every $-N\le k\le N$, the cube
$$
B_k^{\bb{j_2,\cdots,j_{2d}}}:=\bb{k,j_2,j_3,\cdots,j_{2d}}+Q_0,
$$
is a grid-cube. Now, $K\cap B_{-N}^{\bb{j_2,\cdots,j_{2d}}}$ is separated from $K\cap B_{-N+1}^{\bb{j_2,\cdots,j_{2d}}}$ by the hyperplane $H_1^1=\bset{x_1=-N+\frac12}$. Applying Lemma \ref{lem:hyper}, we conclude that the set
$$K\cap\bb{B_{-N}^{\bb{j_2,\cdots,j_{2d}}}\cup B_{-N+1}^{\bb{j_2,\cdots,j_{2d}}}}$$ is a polynomially convex set. Next, the set $K\cap\bb{B_{-N}^{\bb{j_2,\cdots,j_{2d}}}\cup B_{-N+1}^{\bb{j_2,\cdots,j_{2d}}}}$ is separated from the set $K\cap B_{-N+2}^{\bb{j_2,\cdots,j_{2d}}}$ by the hyperplane $H_2^1=\bset{x_1=-N+1+\frac12}$, and as before, applying Lemma \ref{lem:hyper}, the set $$K\cap\bb{B_{-N}^{\bb{j_2,\cdots,j_{2d}}}\cup B_{-N+1}^{\bb{j_2,\cdots,j_{2d}}}\cup B_{-N+2}^{\bb{j_2,\cdots,j_{2d}}}}$$ is polynomially convex. By induction, for every $\bb{j_2,\cdots,j_{2n}}$, the set $$K^{\bb{j_2,\cdots,j_{2d}}}:=K\cap \bunion k {-N}N B_{k}^{\bb{j_2,\cdots,j_{2d}}}$$ is a polynomially convex set.

{\bf Step 2:} For every $\bb{j_3,\cdots,j_{2d}}\in\Z^{2d-2}$ fixed, the set $K^{\bb{-N,j_3,\cdots,j_{2d}}}$ is separated from the set $K^{\bb{-N+1,j_3,\cdots,j_{2d}}}$ by the hyperplane $H_1^2:=\bset{y_1=-N+\frac12}$. Applying Lemma \ref{lem:hyper}, the set $$K^{\bb{-N,j_3,\cdots,j_{2d}}}\cup K^{\bb{-N+1,j_3,\cdots,j_{2d}}}$$ is polynomially convex. As in Step 1, by induction, the set $$K^{\bb{j_3,\cdots,j_{2d}}}:=\bunion k {-N}N K^{\bb{k,j_3,\cdots,j_{2d}}}$$ is polynomially convex.

{\bf Step n:} Assume without loss of generality that $n$ is even, the odd case is identical. For every $\bb{j_{n+1},\cdots,j_{2d}}\in\Z^{2d-n}$ fixed, the sets $K^{\bb{k,j_{n+1},\cdots,j_{2d}}}$ are polynomially convex by the induction assumption on step $(n-1)$. In addition, the set $K^{\bb{-N,j_{n+1},\cdots,j_{2d}}}$ is separated from the set $K^{\bb{-N+1,j_{n+1},\cdots,j_{2d}}}$ by the hyperplane $H_1^n:=\bset{y_{\frac n2}=-N+\frac12}$. If $n$ is odd, the hyperplane will be defined by $\bset{x_{\lfloor\frac n2\rfloor}=-N+\frac12}$. Applying Lemma \ref{lem:hyper}, the set $$K^{\bb{-N,j_{n+1},\cdots,j_{2d}}}\cup K^{\bb{-N+1,j_{n+1},\cdots,j_{2d}}}$$ is polynomially convex. As in Step 1, by induction, the set $$K^{\bb{j_{n+1},\cdots,j_{2d}}}:=\bunion k {-N}N K^{\bb{k,j_{n+1},\cdots,j_{2d}}}$$ is polynomially convex.

After $2d$ steps we obtain that the set $K=K\setminus\underset{q\in\Z^{2d}}\bigcup\partial\bb{q+Q_0}$ is polynomially convex.
\end{proof}


\begin{cor}\label{cor:non-grid}
Let $K\subset\C^{d}$ be a compact set and assume that $K$ does not intersect a finite collection of hyperplanes parallel to the axes. If for every box, $B$, generated by those hyperplanes, the set $K\cap B$ is polynomially convex, then $K$ is polynomially convex.
\end{cor}
The proof is the same as the proof of Proposition \ref{prop:grid} but the notation is more complicated.


\subsection{The proof of Proposition \ref{prop:poly_conv}}
\begin{proof}
To simplify the notation in the proof, we write $\C^d\cong\R^{2d}=\bset{\bb{x_1,x_2,\cdots,x_{2d}}, x_j\in\R}$.

Since the collection of unit cubes is finite there exists $N\in\N$ so that $\bunion j 1 M Q_j\subset\sbb{-N,N}^{2d}$. For every $1\le k\le 2d$ and $-N\le j\le N-1$ we define the hyperplane $H_k^j:=\bset{x_k=j+\frac12}$, and the strip $S_k^j:=\bset{\abs{x_k-\bb{j+\frac12}}<\delta}$, where $\delta=\delta(\eps,M)$ will be defined at the end of the proof. Let $U_0:=\bb{\bunion j 1 M Q_j}\setminus\bb{\bunion k 1 {2d}\bunion j{-N}{N-1} S_k^j}$. Using Proposition \ref{prop:grid}, it is enough to show that there exists a set $U\subset U_0$ so that for every grid-cube, $G$, the set $G\cap U$ is polynomially convex, and $m_{2d}\bb{\bunion j 1 M Q_j\setminus U}<\eps$.

We will modify the set $U_0$ inside each grid-cube. First note that since the grid-cubes and the cubes, $Q_j$, are both closed, parallel to the axes, and have the same edge-length, then the intersection of any two of them is either empty of must contain a vertex of the grid-cube. Since the collection $\bset{Q_j}$ is pairwise disjoint, this means that for every grid-cube, $G$,
$$
\#\bset{j, Q_j\cap G\neq\emptyset}\le \#\bset{v, \text{vertex of }G}=\#\bset{\bset{-\frac12,\frac12}^{2d}}=2^{2d}.
$$
Next, following Lemma \ref{lem:grid}, for every cube, $Q_j$ there exists a grid-cube, which we will denote $G_j$, so that $Q_j\cap G_j$ contains a $\frac1{2^{2d}}$ sub-cube of $Q_j$. Since we saw that $Q_j$ contains the center of $G_j$, every grid-cube contains at most \underline{one} $\frac1{2^{2d}}$ sub-cube of $Q_j$ for some $Q_j$ in the collection.

Fix a grid-cube, $G$.

{\bf Case 1:} $G$ does not contain a $\frac1{2^{2d}}$ sub-cube of any cube in the collection, $\bset{Q_j}$.\\
For every $Q_{j_\nu}$ intersecting $G$, let $\bset{x_k=a_k^{\pm j_\nu}}_{k=1}^{2d}$ denote the collection of hyperplanes defining the cube $Q_{j_\nu}$, i.e., $Q_{j_\nu}=\bintersect k 1 {2d}\bset{a_k^{-j_\nu}\le x_k\le a_k^{+j_\nu}}$. We define
$$
U\cap G:=G\cap U_0\setminus\bb{\underset{j_u}\bigcup\bunion k 1 {2d}\underset{\sigma\in\bset{+,-}}\bigcup \bset{\abs{x_k-a_k^{\sigma j_\nu}}<\delta}},
$$
i.e., we remove from all the other cubes which intersect $G$ a 'fattening' by $2\delta$ of each hyperplane defining $Q_{j\nu}$. The hyperplanes $\bset{x_k=a_k^{\pm j_\nu}}_{k=1}^{2d}$ for $Q_{j_\nu}\cap G\neq\emptyset$ separate the components of $U\cap G$, which are now closed boxes. Each closed box is a polynomially convex set. Following Corollary \ref{cor:non-grid}, the set $U\cap G$ is polynomially convex as well. See Figure \ref{fig:cases} below.

{\bf Case 2:} $G$ contains a $\frac1{2^{2d}}$ sub-cube, i.e., $G=G_j$ for some cube, $Q_j$.\\
Essentially we will repeat the first case, but we first partition $G$ according the hyperplanes defining $Q_j$. Formally, assume $Q_j=\bintersect k 1 {2d}\bset{a_k^-\le x_k\le a_k^+}$, and define $S_k^{\pm}:=\bset{\abs{x_k-a_k^{\pm}}<\delta}$. Then $G\setminus\bunion k 1 {2d} S_k^\pm$ is a union of closed boxes, one of them contains a $\frac1{2^{2d}}$ sub-cube of $Q_j$ that 'lost' at most $2d$ 'fattened' hyperplanes from its boundary.

Next, we will repeat the same argument described in the first case in each box composing the set $G\setminus\bunion k 1 {2d} S_k^\pm$, to construct a set $U$. As in the first case, the set $U$ satisfies that in each of those boxes, $B$, the set $B\cap U$ is polynomially convex. Then we apply Corollary \ref{cor:non-grid} to the grid of hyperplanes giving rise to the cube $Q_j$ to conclude that $U\cap G$ is polynomially convex.

\begin{figure}[ht]
	\centering
	\includegraphics[scale=0.85]{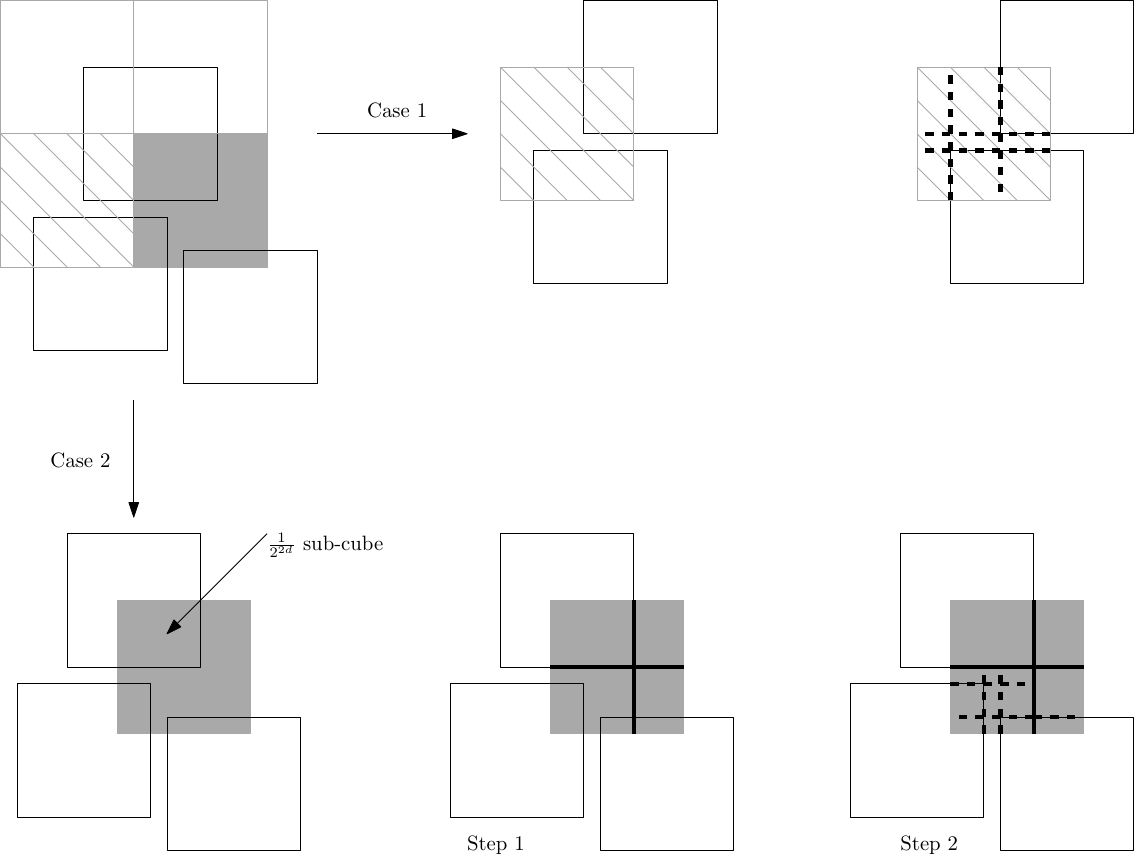}
	\caption{The top left figure shows four grid cubes (whose boundary is dark gray), and three cubes in the collection, $\bset{Q_j}$ (empty cubes with black boundary).\\
	The first case, where the grid cube is striped, is depicted in the top line. The second case, where the grid cube is dark gray, is depicted in the bottom line. The thick black lines are the strips we remove to create the set $U$. As described above, in the second case we have two steps.}
	\label{fig:cases}
\end{figure}

It is left to estimate how much measure we lost. Note that the volume every strip removes for every cube is bounded by $2\delta$ since $S\cap Q_j=q+[-\delta,\delta]\times\sbb{0,1}^{2d-1}$ where $q\in\R^{2d}$ is some translation of the strip. We will first bound how many strips we remove from every cube.

By duality between grid-cubes and cubes in the collection, the number of grid-cubes that intersect $Q_j$ is bounded by $2^{2d}$ as well. Each grid cube is responsible to removing up to $2^{2d}$ strips itself, and up to $2^{2d}\times 2^{2d}$ strips from the cubes in the collection it might intersect ($2^{2d}$ cubes each gives rise to at most $2^{2d}$ strips). Overall,
\begin{align*}
m_{2d}(Q_j\setminus U)&\le 2\delta\cdot\#\bset{S, \text{strip}, S\cap Q_j\neq\emptyset}\\
&\le 2\delta\cdot \#\bset{G,\text{grid-cube}, G\cap Q_j\neq\emptyset}\cdot \#\bset{S, \text{Strip in }G}\le2\delta\cdot 2^{2d}\bb{2^{4d}+2^{2d}}\le \delta\cdot 2^{7d}.
\end{align*}
Overall, we see that
$$
m_{2d}\bb{\bunion j 1 M Q_j\setminus U}\le M\cdot \delta\cdot 2^{7d}=\eps,
$$
if we set $\delta=\delta(\eps,M)=\frac{\eps}{M\cdot  2^{7d}}$. Naturally, if $Q_j$ was a grid cube, then $Q_j^{-\frac{\eps}{M\cdot  2^{7d}}}\subset U$ concluding the proof.
\end{proof}
\begin{rmk}
Note that since the $2d$ Lebesgue measure of hyperplanes is zero, the lemma holds for half open half closed cubes of the form $Q=q+[0,1)^{2d}$ as well.
\end{rmk}

\section{Actions of Arbitrary Free Ergodic Groups}
This section will extend the construction done for the solenoidal group for general ergodic free actions of $\C^d$. The first subsection compiles definitions and known results that will be required for the proof. All of the definitions and proofs can be found in \cite{AG2019} and \cite{W}.
\subsection{Definitions and Preliminaries}
Throughout this section we let $\bb{X,\mathcal B,\mu}$ be a standard probability space and let $T:\C^d\rightarrow PPT(X)$ be an arbitrary free action of the group $\C^{d}$, i.e., for $\mu$-almost every $x\in X$, $T_{\vect z}x=x$ implies that $\vect z=\vect 0$. In other words, the set of all periodic points has measure zero.

\begin{defn}
A map $F:X\rightarrow\C$ is called {\bf measurably entire} if it is a non-constant measurable function and for $\mu$-almost every $x\in X$ the map $F_x:\C\rightarrow\C$ defined by $F_x(z):=F(T_zx)$ is entire. We will call it {\bf measurably entire dense} if in addition $F_x$ has a dense orbit under the action by translations, for almost every $x\in X$.
\end{defn}

We would like to define a function $F_x$ without worrying about the measurability in $x$. We will therefore use the notion of $S$-sets:
\begin{defn}
Let $S\subset\C^d$ be a compact set. A set $B\in\mathcal B$ is called  an {\bf $S$-set} if
\begin{enumerate}[label=($F_{\arabic*}$),leftmargin=0.7cm]
\item For every $\vect z\neq \vect w\in S$, $T_{\vect z}B\cap T_{\vect w}B=\emptyset$.
\item For every $B'\subset B\subset X$ measurable, and every $A\subset S\subset\C^d$ measurable, the set $AB':=\underset{\vect z\in A}\bigcup T_{\vect z}B'$ is a measurable subset of $X$.
\end{enumerate}
\end{defn}

In the definition above, the set $S$ is the set of 'shifts' by the action of $\C^d$, marked $T$, while $B\subset X$ is a $\mu$-null set that we 'shift' by elements of $S$ in the space $X$. While $B$ is $\mu$-null set, we can still identify the set $SB$ with the product space $S\times B$:
\begin{rmk}\label{rmk:JonsObservation}\cite[Remark 2.5]{AG2019}
For every square $S\subset\C$ and every $B\in\mathcal B$, an $S$-set
$$
\mu\bb{AB}:=\mu\bb{\bset{T_{\vect z}B,\; \vect z\in A}}=\frac{m_{2d}(A)}{m_{2d}(S)}\cdot\mu\bb{SB},\;\text{ for every } A\subset S\text{ measurable},
$$
where $m_{2d}$ denotes the $2d$-dimensional Lebesgue measure.
\end{rmk}
We can use this property to define a measure on Borel subsets of the set $B$-
$$
\nu(B')=\frac{\mu\bb{SB'}}{\mu\bb{SB}},
$$
and, in turn, define the product measure on the product space as follows; for every $B'\subset B$ and $A\subset S$ we let
$$
\bb{ \bb{\frac{1}{m_{2d}(S)}\cdot m_{2d}(\cdot)}\times \nu\bb\cdot}(A\times B')=\frac{m_{2d}(A)}{m_{2d}(S)}\cdot \nu(B')=\frac{\mu\bb{AB}}{\mu\bb{SB}}\cdot \frac{\mu\bb{SB'}}{\mu\bb{SB}}.
$$
For our purpose, $S$ will be a $2d$ dimensional cube.

We are interested in $S$-sets, for $S\subset\C^d$ a cube, because these sets allow us to assign for every $x\in B$ a function $F_x:S\rightarrow \C$, which is holomorphic in $S$, creating a measurably entire function $F:SB\rightarrow\C^d$ defined by $F(T_{\vect z}x)=F(\vect z,x)=F_x(\vect z)$ without worrying about inconsistencies in the definition of $F$. Note that as $B$ is an $S$-set, the map $T_{\vect z}x\mapsto (\vect z,x)$ is well defined, and so is our function $F$. We would therefore like to approximate our space $X$ by a sequence of sets $\bset{S_nB_n}_{n=1}^\infty$ where $B_n$ is an $S_n$-set, and  $S_n=\left[0,a_n\right)^{2d}$ for $a_n\nearrow\infty$.

The first lemma is a natural extension of Rokhlin's lemma, that approximates the space $X$ by a sequence of product spaces $S_n\times B_n$ where $B_n\subset X$ and $S_n\subset\C^d$.
\begin{lem}\label{lem:Nested}[The Nested Towers Lemma]
Let $\bset{a_n}_{n=1}^\infty$ be an increasing sequence of positive numbers such that $\sumit n 1 \infty \frac{a_n}{a_{n+1}}<\frac12$. Then there exists a sequence of sets $\bset{B_n}\subset \mathcal B$ such that
\begin{enumerate}[label=($N_\arabic*$),leftmargin=0.7cm]
\item $B_n$ is an $S_n=\left[0,a_n\right)^{2d}$-set.
\item $S_nB_n\subset S_{n+1}B_{n+1}$.
\item $\mu\bb{S_nB_n}\nearrow1$ as $n\rightarrow\infty$.
\end{enumerate}
\end{lem}
A detailed proof can be found for example in the appendix of \cite{AG2019}. While the proof deals with the case of $d=1$ the same proof holds for all $d$.

Let $\bset{B_n}$ be a sequence of sets given by the Nested Towers Lemma, Lemma \ref{lem:Nested}. Assume that $\mathcal P_{n-1}$ is a finite partition of $B_{n-1}$ into measurable sets $\bset{B_{n-1}^j}_{j=1}^{k_{n-1}}$. For every $x\in B_n$ and $1\le\ell\le k_{n-1}$ define the set
$$
R_n^\ell(x):=\bset{\vect z\in S_n;\;T_{\vect z}x\in B_{n-1}^\ell}.
$$
Intuitively, the set $R_n^\ell(x)$ is the set of all points in the cube $S_n$ which will send $x$ back to $B_{n-1}^\ell$. Since $B_{n-1}$ is an $S_{n-1}$-set, for every $\vect z\neq \vect w\in R_n^\ell(x)$ we have that $S_{n-1} T_{\vect z}x\cap S_{n-1} T_{\vect w}x=\emptyset$ which implies  $\norm{\vect z-\vect w}\infty>a_{n-1}$. In particular, $R_n^\ell(x)$ is a finite set and therefore compact.

Following \cite{AG2019}, given $\delta>0$, we say a partition $\mathcal P_n=\bset{B_n^j}_{j=1}^{k_n}$ is a {\bf $\delta$-fine partition consistent with $\mathcal P_{n-1}$} if it is a finite measurable partition of $B_n$, and for every $1\le j\le k_n$ for every $x,y\in B_n^j$, $d_H(R_n^\ell(x),R_n^\ell(y))<\delta$ for every $1\le\ell\le k_{n-1}$, where $d_H$ denotes the Hausdorff distance between sets (see Figure \ref{fig:partition} below).

\begin{figure}[ht]
	\centering
	\includegraphics[scale=0.75]{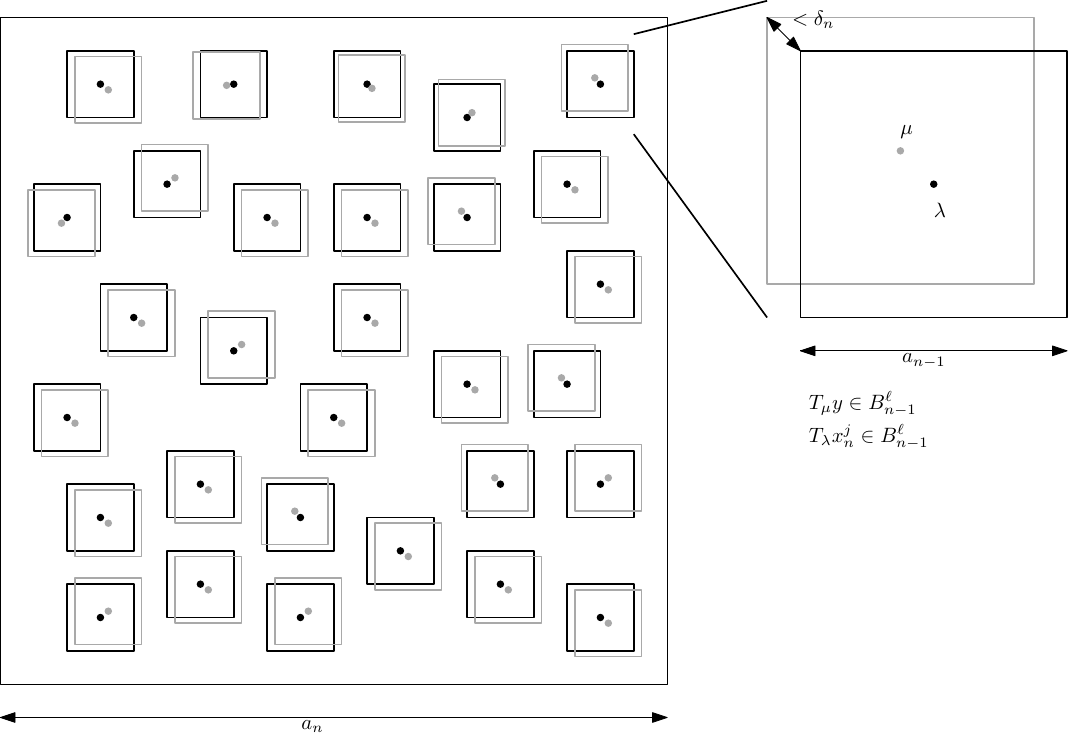}
	\caption{An example of the sets $R_n^\ell(x)$ and $R_n^\ell(y)$ for some $y,x$ in the same partition set.}
	\label{fig:partition}
\end{figure}

\begin{lem}\label{lem:Partition}\cite[Lemma 2.7]{AG2019}
Let $\bset{B_n}$ be a sequence of sets given by the Nested Towers Lemma. For every sequence of positive numbers $\bset{\delta_n}$ there exists a sequence of partitions $\bset{\mathcal P_n}$ such that for every $n$, $\mathcal P_n$ is a $\delta_n$-fine partition consistent with $\mathcal P_{n-1}$.
\end{lem}
As in the case of the Nested Towers Lemma, this lemma was proved for $d=1$ but the same proof extends to higher dimensions.

\subsection{On the Existence of Measurably Entire Dense Functions}
\begin{thm}\label{thm:general}
For every standard probability space, $\bb{X,\mathcal B,\mu}$, and every free action of the group $\C^{d}$, $T:\C^d\rightarrow PPT(X)$, there exists a measurably entire dense function.
\end{thm}
\begin{proof}
Let $\bset{a_n}\nearrow\infty$ be a sequence growing fast enough and denote by $S_n:=\left[0,a_n\right)^{2d}$. Let $\bset{B_n}$ be a sequence of sets obtained for the sequence $\bset{a_n}$ by Weiss' Nested Towers Lemma, Lemma \ref{lem:Nested}, such that $\mu\bb{X\setminus S_1B_1}<\frac 1{200}$.  Let $\bset{p_n}_{n=1}^\infty$ be an enumeration of all polynomials over $\C^d$ with complex-rational coefficients. Note that not only this set is dense in the set of holomorphic functions over $\C^d$, if we remove a finite number of elements, it remains a dense set.

We will set $X_n:=S_nB_n$, and define the sequence of functions $F_n$  as a linear combination of step functions,
$$
F_n\bb{T_{\vect z}x}=F_n\bb{\vect z,x}=\sumit j1{k_n}F_n^j\bb{\vect z}\cdot\indic{B_n^j}(x),
$$
where $\bset{F_n^j}_{j=1}^{k_n}$ are holomorphic, and $\bset{B_n^j}_{j=1}^{k_{n-1}}$ is a measurable partition of $B_n$, denoted $\mathcal P_n$. Note that $F_n$ is well defined, since $B_n$ is an $S_n$-set and therefore the mapping $T_{\vect z}x\mapsto \bb{\vect z,x}$ is well defined. Each such function will satisfy:
\begin{align*}
\tag {Condition (D)}\label{cond:dense}&\text{ For every }1\le k\le n\text{ there exists a translation }\vect w_k^{j,n}(S)\in\C^d \text{ so that}\\
&\underset{\vect z\in Q_k^{-\sumit \ell {k+1}n2^{-\ell}}}\sup\abs{F_n^j\bb{\vect w_k^{j,n}+\vect z}-p_k\bb{\vect z}}<\sumit \ell {k+1}n2^{-\ell},\text{ where } Q_k \text{ is a }\frac1{2^{2d}} \text{ sub-cube of }S_k,
\end{align*}
where given a set $U$ and $\eps>0$ we define
\begin{align*}
&U^{+\eps}:=\bset{\vect z\in \C^d, dist(\vect z,U)<\eps}\\
&U^{-\eps}:=\bset{\vect z\in U, dist(\vect z,\C\setminus U)>\eps}.
\end{align*}
Intuitively $U^{+\eps}$ is a 'fattening' of $U$ by $\eps$ and $U^{-\eps}$ is 'shrinking' $U$ by $\eps$.

Formally, we will construct this sequence inductively.

 Define $F_1:X_1\rightarrow\C$ by
$$
F_1\bb{T_{\vect z}x}=F_1\bb{x,\vect z}=p_1\bb{\vect z}.
$$
$F_1$ is measurable, since it is constant with respect to one variable, and continuous with respect to the other, and it satisfies \ref{cond:dense}.

Assume that $F_{n-1}: X_{n-1}\rightarrow\C$ was defined as
$$
F_{n-1}\bb{T_{\vect z}x}=F_{n-1}\bb{x,\vect z}=\sumit j 1 {k_{n-1}}F_{n-1}^j\bb{\vect z}\indic{B_{n-1}^j}(x),
$$
satisfied \ref{cond:dense}, and
\begin{align}\label{eq:consistent}
\mu\bb{x\in X_{n-2}, \abs{F_{n-1}(x)-F_{n-2}(x)}>2^{-(n-1)}}<2^{-(n-1)}.
\end{align}

Since $F_{n-1}^j$ is holomorphic in $\C^d$ for every $j$ fixed, it is uniformly continuous on $S_{n-1}^{+1}=\left[-1,a_{n-1}+1\right)^{2d}$, and therefore there exists $\delta_n\in(0,1)$ such that for every $1\le j\le k_{n-1}$:
$$
\underset{\vect z,\vect w\in S_{n-1}^{+1}\atop\abs{\vect z-\vect w}<\delta_n}\sup\; \abs{F_{n-1}^j\bb{\vect z}-F_{n-1}^j\bb{\vect w}}<\frac{10^{-2n}}2.
$$
Let $\mathcal P_n$ be a partition of $B_n$ which is $\delta_n$-fine and consistent with $\mathcal P_{n-1}=\bset{B_{n-1}^\ell}_{\ell=1}^{k_{n-1}}$, the partition of $B_{n-1}$ used to define $F_{n-1}$. Such a partition exists by Lemma \ref{lem:Partition}. \\
For every $j$ we use the axiom of choice to choose a representative $x_n^j\in B_n^j$. We associate every point $x_n^j$ with a collection of cubes. First, we note that the Nested Tower Lemma naturally relates for every $x\in B_n$ a collection of copies of $S_{n-1}$ inside $S_n$ defined by
$$
\Lambda_n^j:=\bset{\lambda\in S_n, T_\lambda x_n^j\in B_{n-1}}.
$$
Since $B_n$ is an $S_n$-set, the collection of cubes $\bset{\lambda+S_{n-1}}_{\lambda\in \Lambda_n^j}$ is pairwise disjoint. We will choose the sequence $\bset{a_n}$ so that $\frac{a_n}{a_{n-1}}\in\N$. Before applying Proposition \ref{prop:poly_conv} to this collection, we want to 'clear out' one cube in the center of each $\frac1{2^{2d}}$ sub-cube of $S_n$ to make space for the polynomial $p_n$. Assume, for example, that $a_n=2^{m_n}$ for some sequence $m_n\in\N$ satisfying $m_n\nearrow\infty$. Then $\frac{a_n}{a_{n-1}}=2^{m_n-m_{n-1}}$ and in the center of each of the $\frac1{2^{2d}}$ sub-cubes there exists a copy of $S_{n-1}$ (indeed the choice is not unique, there are $2^{2d}$ options). Denote the centers of the chosen cubes by $\vect w_k^n, 1\le k\le 2^{2d}$. Note that this choice does not depend on the point $x_n^j$ only on the ratio $\frac{a_n}{a_{n-1}}$. In addition, it is important to choose one sub-cube in each $\frac1{2^{2d}}$ of the sub-cubes of $S_n$ since Proposition \ref{prop:poly_conv} only guarantees one of the $\frac1{2^{2d}}$ sub-cubes will be almost fully included in the almost polynomially convex set, $U$.

Define the collection of cubes
$$
\mathcal P_n^j:=\bset{\vect w_k^n+S_{n-1}}_{k=1}^{2^{2d}}\cup\bset{\lambda+S_{n-1}, \bb{\lambda+S_{n-1}}\cap\bb{\bunion k 1 {2^{2d}} \bb{\vect w_k^n+S_{n-1}}}=\emptyset, \lambda\in \Lambda_n^j}.
$$
This collection is a finite collection of disjoint cubes all contained in $S_n$. We can therefore apply Proposition \ref{prop:poly_conv} to obtain a polynomially convex set, $U_n^j\subset S_n$, satisfying that
\begin{enumerate}
\item For every $1\le k\le 2^{2d}$, $\vect w_k^n+S_{n-1}^{-2^{-n}}\subset U_n^j$, where $S_{n-1}^{-2^{-n}}=\left[2^{-n},a_{n-1}-2^{-n}\right)^{2d}$, as $\vect w_k^n+S_{n-1}$ are grid cubes.
\item For every $Q\in\mathcal P_n^j$ there exists a $\frac1{2^{2d}}$ sub-cube, $B_Q$ satisfying $B_Q^{-2^{-n}}\subset U_n^j$.
\item $\frac{m_{2d}\bb{\bb{\underset{Q\in \mathcal P_n^j}\bigcup Q}\setminus U_n^j}}{m_{2d}(S_{n-1})}<2^{-n}$.
\end{enumerate}
Define the map $j_n:B_n\rightarrow\N$ by $j_n(x)=j$ if $x\in B_n^j$, and the function
$$
G_n^j(\vect z)= \begin{cases}
				p_n\bb{\vect z-\vect w_k^n}&, \vect z\in\vect w_k^n+S_{n-1}\\
				F_{n-1}^{j_{n-1}\bb{T_\lambda x_n^j}}\bb{\vect z-\lambda}&, \vect z\in\lambda+S_{n-1}
			\end{cases}.
$$
We use {\bf Oka-Weil Approximation Theorem} to approximate $G_n^j$ on the polynomially convex set $U_n^j$ with error $2^{-n}$ and denote the approximation $F_n^j$. Since $U_n^j$ contains $\frac1{2^{2d}}$ sub-cubes of $S_{n-1}$, excluding a- $2^{-n}$  strip along its boundary, each of them contain $\frac1{2^{2d}}$ sub-cubes of $S_{n-2}$ excluding a $2^{-(n-1)}$ strip along its boundary and so on. By induction, $U_n^j$ contains copies of $\frac1{2^{2d}}$ sub-cubes of $S_k$ excluding a $\sumit m k n 2^{-m}$ strips on their boundaries. In particular $F_n^j$ satisfies \ref{cond:dense}.

We define $F_n:X_n\rightarrow\C$ by
$$
F_{n}\bb{T_{\vect z}x}=F_{n}\bb{x,\vect z}=\sumit j 1 {k_{n}}F_{n}^j\bb{\vect z}\indic{B_{n}^j}(x).
$$
Finally, let $E_n$ be the set where either $F_{n-1}$ was not defined or $F_n$ does not approximate $F_{n-1}$ well, namely
$$
E_n=X\setminus X_{n-1}\cup\bset{x\in X_{n-1}, \abs{F_n(x)-F_{n-1}(x)}>2^{-n}}.
$$
Then, following Remark \ref{rmk:JonsObservation}
\begin{align*}
\mu\bb{E_n}&\le \mu\bb{\bunion j 1 {k_n}\bb{\bb{\underset{Q\in \mathcal P_n^j}\bigcup Q\setminus U_n^j} B_n^j}}+O\bb{\frac{a_{n-1}}{a_n}}+\mu\bb{X\setminus X_{n-1}}\\
&=\sumit j 1 {k_n}\frac{m_{2d}\bb{\bb{\underset{Q\in \mathcal P_n^j}\bigcup Q}\setminus U_n^j}}{m_{2d}\bb{\underset{Q\in \mathcal P_n^j}\bigcup Q}}\cdot \mu\bb{\bb{\underset{Q\in \mathcal P_n^j}\bigcup Q}B_n^j}+O\bb{\frac{a_{n-1}}{a_n}}+\mu\bb{X\setminus X_{n-1}}\\
&\le 2^{-n}\mu\bb{S_nB_n}+O\bb{\frac{a_{n-1}}{a_n}}+\mu\bb{X\setminus X_{n-1}}<2^{-n}
\end{align*}
where $O\bb{\frac{a_{n-1}}{a_n}}$ comes from the fact that we replace up to $\bb{2^{2d}}^2$ elements in the sets $\Lambda_n^j$ to create the collection $\bset{\vect w_k^n+S_{n-1}}$. We conclude that \eqref{eq:consistent} holds.

Lastly, let $F=\limit n\infty F_n$, then using Borel-Cantelli's Lemma, 
$$
\mu\bb{\bset{x\in X \text{ s.t. the orbit of }F_x\atop \text{does not have a dense orbit}} }\le \mu\bb{\bintersect n 1 \infty \bunion m n\infty E_m}=0,
$$
since $\sumit m 1 \infty \mu\bb{E_m}\le \sumit m 1 \infty 2^{-m}<\infty,$ concluding the proof of Theorem \ref{thm:general}.
\end{proof}

\bigskip{\footnotesize%
  \bigskip
  (B.~Weiss) \textsc{Einstein Institute of Mathematics, Hebrew University of Jerusalem, Jerusalem ,Israel} \par
  \textit{E-mail address}: \texttt{weiss@math.huji.ac.il}
  
  \bigskip
  (A.~Gl\"ucksam) \textsc{Einstein Institute of Mathematics, Hebrew University of Jerusalem, Jerusalem ,Israel} \par
  \textit{E-mail address}: \texttt{adi.glucksam@math.huji.ac.il}
}

\end{document}